\documentclass[a4paper,USenglish,cleveref, autoref, thm-restate]{lipics-v2021}

\usepackage{cncartstyle}

\newcounter{numTableRows}

\providetoggle{usePrecompiledGraphs}
\toggletrue{usePrecompiledGraphs}

\tikzstyle{smallgraph}=[
  every node/.style={circle,draw,fill=black!30,inner sep=0pt,minimum width=2.5pt,font={\footnotesize}},
]

\newcommand{\adrawingFig}{
	\begin{figure}
		\centering
		\tikzstyle{origE} = [
			draw=red!80!black
		]
		\tikzstyle{copyE} = [
			densely dotted,
			draw=blue,
			thick,
		]
		\tikzstyle{copyEleft} = [
			copyE,
			bend left=40
			looseness=0.7,
		]
		\tikzstyle{copyEright} = [
			copyE,
			bend right=40
			looseness=0.7,
		]
		\begin{subfigure}[t]{.49\textwidth}
			\centering
			\begin{tikzpicture}[
					smallgraph,
					every node/.style={circle,draw,fill=black!30,inner sep=0pt,minimum width=4pt,font={\footnotesize}},
					every path/.append style={semithick}
				]
				\foreach \x [evaluate=\x as \xShift using \x*1.1] in {0,1,2,3,4} {
					\foreach \n/\b [evaluate=\b as \bw using \b*0.45] in {1/1.5, 2/0.5, 3/-0.5, 4/-1.5}{
						\node (\x-\n) at (\xShift,\bw) {};
					}

					\draw[origE] (\x-1) -- (\x-2) -- (\x-3) -- (\x-4);

					\path ($(\x-2)!.5!(\x-3)$) -- ++(90:1.0cm) node[draw=none,fill=none,minimum width=0pt] () {};%
					\path ($(\x-2)!.5!(\x-3)$) -- ++(270:1.0cm) node[draw=none,fill=none,minimum width=0pt] () {};%
				};

				\draw[origE,bend right=75,looseness=0.7] (0-1) to (0-4);
				\draw[origE,bend right=35] (0-2) to (0-4);
				\draw[origE,bend left=50] (0-1) to (0-3);
				\draw[origE,bend left=75,looseness=0.7] (4-1) to (4-4);
				\draw[origE,bend right=50] (4-2) to (4-4);
				\draw[origE,bend left=35] (4-1) to (4-3);

				\draw[origE,rounded corners=.25cm] (3-1) -- ($(4-1)!.5!(3-1) + (0,.2)$)
					-- ($(4-1)+(.45,.2)$) -- ($(4-4)+(.45,-.2)$)
					-- ($(4-4)!.5!(3-4) + (0,-.2)$) -- (3-4);
				\draw[origE,bend left=50] (3-1) to (3-3);
				\draw[origE,bend right=50] (3-2) to (3-4);

				\foreach \x [evaluate=\x as \outerX using \x*0.15] in {1,2} {
					\draw[origE,rounded corners=.25cm] (\x-1) -- ($(0-1)!.5!(\x-1) + (0,0.05+\outerX)$)
						-- ($(0-1)+(-0.3-\outerX,0.05+\outerX)$) -- ($(0-4) - (0.3+\outerX,0.05+\outerX)$)
						-- ($(0-4)!.5!(\x-4) - (0,0.05+\outerX)$) -- (\x-4);
					\draw[origE,bend left=50] (\x-1) to (\x-3);
					\draw[origE,bend right=50] (\x-2) to (\x-4);
				}

				\foreach \x [remember=\x as \lastx (initially=0)] in {1,2,3,4} {
					\draw[copyE] (\lastx-1) to (\x-1);
					\draw[copyE] (\lastx-2) to (\x-2);
					\draw[copyE] (\lastx-3) to (\x-3);
					\draw[copyE] (\lastx-4) to (\x-4);

				}
			\end{tikzpicture}
			\caption{Optimal drawing with $8$ crossings; the three inner $K_4$-copies are involved in $2$ crossings each.}
			\label{fig:cartesianproduct}
		\end{subfigure}
		\hfill
		\begin{subfigure}[t]{.49\textwidth}
			\centering
			\begin{tikzpicture}[
					smallgraph,
					every node/.style={circle,draw,fill=black!30,inner sep=0pt,minimum width=4pt,font={\footnotesize}},
					every path/.append style={semithick}
				]
				\foreach \x [evaluate=\x as \xShift using \x*1.5] in {0,2,4} {
					\foreach \n/\a/\b [evaluate=\a as \ah using \a*0.35,evaluate=\b as \bw using \b*0.4] in {
						1/1/1.0,
						2/-1/1,
						3/-1/-1.0,
						4/1/-1
					}{
						\node (\x-\n) at ([xshift=+\xShift cm] \ah,\bw) {};
					}

					\draw[origE] (\x-2) -- (\x-3) -- (\x-4) -- (\x-1) -- (\x-2) -- (\x-4) -- (\x-1) -- (\x-3);

					\path ($(\x-2)!.5!(\x-3)$) -- ++(90:1.0cm) node[draw=none,fill=none,minimum width=0pt] () {};%
					\path ($(\x-2)!.5!(\x-3)$) -- ++(270:1.0cm) node[draw=none,fill=none,minimum width=0pt] () {};%
				};
                \foreach \x [evaluate=\x as \xShift using \x*1.5] in {1,3} {
					\foreach \n/\a/\b [evaluate=\a as \ah using \a*0.35,evaluate=\b as \bw using \b*0.4] in {
						1/-1/1.0,
						2/1/1,
						3/1/-1.0,
						4/-1/-1
					}{
						\node (\x-\n) at ([xshift=+\xShift cm] \ah,\bw) {};
					}

					\draw[origE] (\x-2) -- (\x-3) -- (\x-4) -- (\x-1) -- (\x-2) -- (\x-4) -- (\x-1) -- (\x-3);

					\path ($(\x-2)!.5!(\x-3)$) -- ++(90:1.0cm) node[draw=none,fill=none,minimum width=0pt] () {};%
					\path ($(\x-2)!.5!(\x-3)$) -- ++(270:1.0cm) node[draw=none,fill=none,minimum width=0pt] () {};%
				};
				\foreach \x [remember=\x as \lastx (initially 0)] in {1,2,3,4} {
					\draw[copyE, bend left=30, looseness=1.6] (\lastx-1) to (\x-1);
					\draw[copyE, bend left=30, looseness=1.6] (\lastx-2) to (\x-2);
					\draw[copyE, bend right=30, looseness=1.6] (\lastx-3) to (\x-3);
					\draw[copyE, bend right=30, looseness=1.6] (\lastx-4) to (\x-4);

				}
			\end{tikzpicture}
			\caption{\adrawing~for $a=2$ with $11$ crossings.}
			\label{fig:adrawing}
		\end{subfigure}
		\caption{Two drawings of $K_4 \Box P_4$. The $K_4$-copies are solid red; the path edges are dotted blue.}
	\end{figure}
}

\newcommand{\resultsTblLegend}{
	\label{tab:key}
	\begin{tabular}{ccc}
		\toprule
		symbol             & previously known           & solved by us       \\
		\midrule
		\statusNew         & no                         & yes                \\[.0em]
		\statusConfStar    & not peer-reviewed          & yes                \\[.0em]
		\statusConf        & yes                        & yes                \\[.0em]
		\statusFail        & yes                        & no                 \\[.0em]
		\statusNobodyKnows & no                         & no                 \\[.0em]
		\bottomrule
	\end{tabular}
}

\newcommand{\forcesFig}{\includegraphics[page=1]{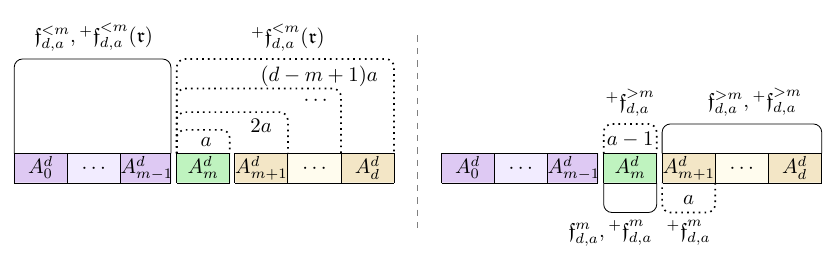}}
\newcommand{\crossingBandsTbl}{\includegraphics[page=2]{results.pdf}}
\newcommand{\numcrossingsFig}{\includegraphics[page=3]{results.pdf}}
\newcommand{\resultsFiveTbl}{\includegraphics[page=4]{results.pdf}}
\newcommand{\resultsSixTbl}{
   \includegraphics[page=5]{results.pdf}
   \includegraphics[page=6]{results.pdf}
   \includegraphics[page=7]{results.pdf}
}
\newcommand{\resultsSevenTbl}{
   \includegraphics[page=8]{results.pdf}
   \includegraphics[page=9]{results.pdf}
   \includegraphics[page=10]{results.pdf}
   \includegraphics[page=11]{results.pdf}
   \includegraphics[page=12]{results.pdf}
   \includegraphics[page=13]{results.pdf}
   \includegraphics[page=14]{results.pdf}
   \includegraphics[page=15]{results.pdf}
   \includegraphics[page=16]{results.pdf}
   \includegraphics[page=17]{results.pdf}
   \includegraphics[page=18]{results.pdf}
   \includegraphics[page=19]{results.pdf}
   \includegraphics[page=20]{results.pdf}
   \includegraphics[page=21]{results.pdf}
}

\newcommand{\specialFortyTwo}{
    \begin{tikzpicture}[smallgraph]
    \definecolor{col0}{rgb}{0.6249089169450681, 0.0009573646476685411, 0.7035877632862583}
    \definecolor{col1}{rgb}{0.0, 0.8, 0.0}
    \definecolor{col2}{rgb}{0.75, 0.375, 0.0}
    \definecolor{col3}{rgb}{0.0006151721264259039, 0.3939063320982763, 0.7164887572747496}
    \definecolor{col4}{rgb}{0.25108805022687397, 0.6941457781079515, 0.4433085415751947}
    \definecolor{col5}{rgb}{0.13992995297254043, 0.000000000000000001, 0.5648521160561512}
    \definecolor{col6}{rgb}{0.1610299703780087, 0.37622735282779585, 0.2479456574066141}
    	\draw ($(-0.375,0.0)+(-0.75,0)$) node[opacity=0.0] (3f) {};
    	\draw ($(-0.375,0.16)+(-0.75,0)$) node[opacity=0.0] (0f) {};
    	\draw ($(-0.375,0.32)+(-0.75,0)$) node[opacity=0.0] (5f) {};
    	\draw ($(-0.375,0.48)+(-0.75,0)$) node[opacity=0.0] (2f) {};
    	\draw ($(-0.375,0.64)+(-0.75,0)$) node[opacity=0.0] (4f) {};
    	\draw ($(-0.375,0.8)+(-0.75,0)$) node[opacity=0.0] (1f) {};
    	\draw ($(0.018666666666666665,0.0)+(-0.75,0)$) node[opacity=0.5] (3mm) {};
    	\draw ($(0.018666666666666665,0.16)+(-0.75,0)$) node[opacity=0.5] (0mm) {};
    	\draw ($(0.042666666666666665,0.32)+(-0.75,0)$) node[opacity=0.5] (5mm) {};
    	\draw ($(-0.06133333333333334,0.48)+(-0.75,0)$) node[opacity=0.5] (2mm) {};
    	\draw ($(-0.06133333333333334,0.64)+(-0.75,0)$) node[opacity=0.5] (4mm) {};
    	\draw ($(0.042666666666666665,0.8)+(-0.75,0)$) node[opacity=0.5] (1mm) {};
    	\draw (0.018666666666666665,0.0) node[] (3m) {};
    	\draw (0.018666666666666665,0.16) node[] (0m) {};
    	\draw (0.042666666666666665,0.32) node[] (5m) {};
    	\draw (-0.06133333333333334,0.48) node[] (2m) {};
    	\draw (-0.06133333333333334,0.64) node[] (4m) {};
    	\draw (0.042666666666666665,0.8) node[] (1m) {};
    	\draw (0.7686666666666666,0.0) node[] (3e) {};
    	\draw (0.7686666666666666,0.16) node[] (0e) {};
    	\draw (0.7926666666666666,0.36) node[] (5e) {};
    	\draw (0.6886666666666666,0.48) node[] (2e) {};
    	\draw (0.6886666666666666,0.64) node[] (4e) {};
    	\draw (0.7926666666666666,0.76) node[] (1e) {};
    	\draw (0,1.008) node[fill=none,draw=none] (spacerTop) {};
    	\draw (0,-0.52800000000000002) node[fill=none,draw=none] (spacerBot) {};
    	\draw (-1.000,0) node[fill=none,draw=none] (spacerLeft) {};
    	\draw (1.125,0) node[fill=none,draw=none] (spacerRight) {};
    	\draw[densely dotted, opacity=0.5] (3f) to (3mm);
    	\draw[densely dotted, opacity=0.5, absolute, out=0, in=180, looseness=1.7] (0f) to (0mm);
    	\draw[densely dotted, opacity=0.5] (5f) to (5mm);
    	\draw[densely dotted, opacity=0.5] (2f) to (2mm);
    	\draw[densely dotted, opacity=0.5, absolute, out=0, in=180, looseness=1.7] (4f) to (4mm);
    	\draw[densely dotted, opacity=0.5] (1f) to (1mm);
    	\draw[opacity=0.5, col2] (3mm) to (0mm);
    	\draw[opacity=0.5, col1] (0mm) to (5mm);
    	\draw[opacity=0.5, col0, bend right=52, looseness=1.0666666666666667] (5mm) to (1mm);
    	\draw[opacity=0.5, col6] (5mm) to (2mm);
    	\draw[opacity=0.5, col3] (2mm) to (4mm);
    	\draw[opacity=0.5, col4] (4mm) to (1mm);
    	\draw[opacity=0.5, densely dotted] (3mm) to (3m);
    	\draw[opacity=0.5, densely dotted] (0mm) to (0m);
    	\draw[opacity=0.5, densely dotted] (5mm) to (5m);
    	\draw[opacity=0.5, densely dotted] (2mm) to (2m);
    	\draw[opacity=0.5, densely dotted] (4mm) to (4m);
    	\draw[opacity=0.5, densely dotted] (1mm) to (1m);
    	\draw[col2] (3m) to (0m);
    	\draw[col1] (0m) to (5m);
    	\draw[col0, thick, rounded corners, looseness=1.0666666666666667] (5m) -- ($(5m)+(0.2,-0.08)$) -- ($(5e)+(0.2,-0.1)$) -- ($(1e)+(0.2,0.1)$) -- ($(1m)+(0.2,0.08)$) -- (1m);
    	\draw[col6] (5m) to (2m);
    	\draw[col3] (2m) to (4m);
    	\draw[col4] (4m) to (1m);
    	\draw[densely dotted] (3m) to (3e);
    	\draw[densely dotted] (0m) to (0e);
    	\draw[densely dotted] (5m) to (5e);
    	\draw[densely dotted] (2m) to (2e);
    	\draw[densely dotted] (4m) to (4e);
    	\draw[densely dotted] (1m) to (1e);
    	\draw[col2] (3e) to (0e);
    	\draw[col1] (0e) to (5e);
    	\draw[col0, bend right=52, looseness=1.0666666666666667] (5e) to (1e);
    	\draw[col6] (5e) to (2e);
    	\draw[col3] (2e) to (4e);
    	\draw[col4] (4e) to (1e);
     \draw (0.0, -0.4) node[fill=white,draw=white,rectangle,inner sep=1pt] (cap) {Two last copies ($G^6_{42}$)};
    \end{tikzpicture}
}

\newcommand{\specialSixtyThree}{
    \begin{tikzpicture}[smallgraph]
    \definecolor{col0}{rgb}{0.0, 0.8, 0.0}
    \definecolor{col1}{rgb}{0.7, 0.0, 0.7}
    \definecolor{col2}{rgb}{0.0, 0.375, 0.75}
    \definecolor{col3}{rgb}{0.75, 0.375, 0.0}
    \definecolor{col4}{rgb}{0.36250000000000004, 0.5437500000000001, 0.36250000000000004}
    \definecolor{col5}{rgb}{0.25489766271245545, 0.008014283176777229, 0.5424609659596662}
    \definecolor{col6}{rgb}{0.6652640214649158, 0.010893464020492842, 0.10055778442684428}
    	\draw ($(-0.375,0.0)+(-0.75,0)$) node[opacity=0.0] (4f) {};
    	\draw ($(-0.375,0.16)+(-0.75,0)$) node[opacity=0.0] (2f) {};
    	\draw ($(-0.375,0.32)+(-0.75,0)$) node[opacity=0.0] (0f) {};
    	\draw ($(-0.375,0.48)+(-0.75,0)$) node[opacity=0.0] (5f) {};
    	\draw ($(-0.375,0.64)+(-0.75,0)$) node[opacity=0.0] (1f) {};
    	\draw ($(-0.375,0.8)+(-0.75,0)$) node[opacity=0.0] (3f) {};
    	\draw ($(0.024,0.00)+(-0.75,0)$) node[opacity=0.5] (4mm) {};
    	\draw ($(-0.10400000000000001,0.16)+(-0.75,0)$) node[opacity=0.5] (2mm) {};
    	\draw ($(0.10400000000000001,0.32)+(-0.75,0)$) node[opacity=0.5] (0mm) {};
    	\draw ($(-0.024,0.48)+(-0.75,0)$) node[opacity=0.5] (5mm) {};
    	\draw ($(-7.401486830834376e-19,0.64)+(-0.75,0)$) node[opacity=0.5] (1mm) {};
    	\draw ($(-7.401486830834376e-19,0.8)+(-0.75,0)$) node[opacity=0.5] (3mm) {};
    	\draw (0.234,0.19) node[] (4m) {};
    	\draw (-0.10400000000000001,0.08) node[] (2m) {};
    	\draw (0.10400000000000001,0.32) node[] (0m) {};
    	\draw (-0.024,0.48) node[] (5m) {};
    	\draw (-7.401486830834376e-19,0.64) node[] (1m) {};
    	\draw (-7.401486830834376e-19,0.8) node[] (3m) {};
    	\draw (0.546,0.19) node[] (4e) {};
    	\draw (0.854,0.12) node[] (2e) {};
    	\draw (0.646,0.32) node[] (0e) {};
    	\draw (0.774,0.48) node[] (5e) {};
    	\draw (0.75,0.64) node[] (1e) {};
    	\draw (0.75,0.8) node[] (3e) {};
    	\draw (0,1.008) node[fill=none,draw=none] (spacerTop) {};
    	\draw (0,-0.52800000000000002) node[fill=none,draw=none] (spacerBot) {};
    	\draw (-0.375,0) node[fill=none,draw=none] (spacerLeft) {};
    	\draw (1.125,0) node[fill=none,draw=none] (spacerRight) {};
    	\draw[densely dotted, thick, opacity=0.5, absolute, out=0, in=180, looseness=1.2] (4f) to (4mm);
    	\draw[densely dotted, opacity=0.5, absolute, out=0, in=180, looseness=1.7] (2f) to (2mm);
    	\draw[densely dotted, opacity=0.5] (0f) to (0mm);
    	\draw[densely dotted, opacity=0.5] (5f) to (5mm);
    	\draw[densely dotted, opacity=0.5, absolute, out=0, in=180, looseness=1.7] (1f) to (1mm);
    	\draw[densely dotted, opacity=0.5] (3f) to (3mm);
    	\draw[opacity=0.5, col5, bend right=38, looseness=1.0333333333333334] (4mm) to (0mm);
    	\draw[opacity=0.5, col3] (4mm) to (2mm);
    	\draw[opacity=0.5, col6] (2mm) to (0mm);
    	\draw[opacity=0.5, col6, bend left=38, looseness=1.0333333333333334] (2mm) to (5mm);
    	\draw[opacity=0.5, col1] (0mm) to (5mm);
    	\draw[opacity=0.5, col0] (5mm) to (1mm);
    	\draw[opacity=0.5, col1] (1mm) to (3mm);
    	\draw[thick, densely dotted, absolute, out=-10, in=-100, looseness=0.8] (4mm) to (4m);
    	\draw[opacity=0.5, densely dotted] (2mm) to (2m);
    	\draw[opacity=0.5, densely dotted] (0mm) to (0m);
    	\draw[opacity=0.5, densely dotted] (5mm) to (5m);
    	\draw[opacity=0.5, densely dotted] (1mm) to (1m);
    	\draw[opacity=0.5, densely dotted] (3mm) to (3m);
    	\draw[col5] (4m) to (0m);
    	\draw[col3] (4m) to (2m);
    	\draw[col6] (2m) to (0m);
    	\draw[col6, bend left=38, looseness=1.0333333333333334] (2m) to (5m);
    	\draw[col1] (0m) to (5m);
    	\draw[col0] (5m) to (1m);
    	\draw[col1] (1m) to (3m);
    	\draw[densely dotted, thick] (4m) to (4e);
    	\draw[densely dotted, thick, absolute, out=-10, in=190] (2m) to (2e);
    	\draw[densely dotted] (0m) to (0e);
    	\draw[densely dotted] (5m) to (5e);
    	\draw[densely dotted] (1m) to (1e);
    	\draw[densely dotted] (3m) to (3e);
    	\draw[col5] (4e) to (0e);
    	\draw[col3] (4e) to (2e);
    	\draw[col6] (2e) to (0e);
    	\draw[col6, bend right=38, looseness=1.0333333333333334] (2e) to (5e);
    	\draw[col1] (0e) to (5e);
    	\draw[col0] (5e) to (1e);
    	\draw[col1] (1e) to (3e);
     \draw (0.0, -0.4) node[fill=white,draw=white,rectangle,inner sep=1pt] (cap) {Two last copies ($G^6_{63}$)};
    \end{tikzpicture}
}

\newcommand{\specialEightyTwo}{
    \begin{tikzpicture}[smallgraph]
    \definecolor{col0}{rgb}{0.0, 0.8, 0.0}
    \definecolor{col1}{rgb}{0.7, 0.0, 0.7}
    \definecolor{col2}{rgb}{0.0, 0.375, 0.75}
    \definecolor{col3}{rgb}{0.75, 0.375, 0.0}
    \definecolor{col4}{rgb}{0.36250000000000004, 0.5437500000000001, 0.36250000000000004}
    \definecolor{col5}{rgb}{0.26182363941018166, 0.022576030499820685, 0.533127446741137}
    \definecolor{col6}{rgb}{0.618946624832486, 0.003123796557973679, 0.10732816870139572}
    	\draw (-0.375,0.32) node[opacity=0.0] (0f) {};
    	\draw (-0.375,0.48) node[opacity=0.0] (5f) {};
    	\draw (-0.375,0.64) node[opacity=0.0] (2f) {};
    	\draw (-0.375,0.8) node[opacity=0.0] (1f) {};
    	\draw (-0.375,0.16) node[opacity=0.0] (3f) {};
    	\draw (-0.375,0.0) node[opacity=0.0] (4f) {};
    	\draw (0.043132860308077635,0.0) node[] (4m) {};
    	\draw (-0.1180878210877411,0.16) node[] (3m) {};
    	\draw (0.03650544331010812,0.32) node[] (0m) {};
    	\draw (0.03319173481112335,0.48) node[] (5m) {};
    	\draw (-0.07080826518887666,0.64) node[] (2m) {};
    	\draw (0.03319173481112335,0.8) node[] (1m) {};
    	\draw ($(-0.043132860308077635,0.0)+(0.75,0)$) node[] (4e) {};
    	\draw ($(0.1180878210877411,0.16)+(0.75,0)$) node[] (3e) {};
    	\draw ($(-0.03650544331010812,0.32)+(0.75,0)$) node[] (0e) {};
    	\draw ($(-0.03319173481112335,0.48)+(0.75,0)$) node[] (5e) {};
    	\draw ($(0.07080826518887666,0.64)+(0.75,0)$) node[] (2e) {};
    	\draw ($(-0.03319173481112335,0.8)+(0.75,0)$) node[] (1e) {};
    	\draw ($(0.7712613109346991,0.0)+(0.75,0)$) node[opacity=0.5] (4ee) {};
    	\draw ($(0.7851177173952502,0.16)+(0.75,0)$) node[opacity=0.5] (3ee) {};
    	\draw ($(0.6341242259448514,0.32)+(0.75,0)$) node[opacity=0.5] (0ee) {};
    	\draw ($(0.8128305303163521,0.48)+(0.75,0)$) node[opacity=0.5] (5ee) {};
    	\draw ($(0.6912613109346991,0.64)+(0.75,0)$) node[opacity=0.5] (2ee) {};
    	\draw ($(0.7952613109346991,0.8)+(0.75,0)$) node[opacity=0.5] (1ee) {};
    	\draw (0,1.008) node[fill=none,draw=none] (spacerTop) {};
    	\draw (0,-0.52800000000000002) node[fill=none,draw=none] (spacerBot) {};
    	\draw (-0.375,0) node[fill=none,draw=none] (spacerLeft) {};
    	\draw (1.125,0) node[fill=none,draw=none] (spacerRight) {};
    	\draw[densely dotted, opacity=0.5] (0f) to (0m);
    	\draw[densely dotted, opacity=0.5] (5f) to (5m);
    	\draw[densely dotted, opacity=0.5] (2f) to (2m);
    	\draw[densely dotted, opacity=0.5] (1f) to (1m);
    	\draw[densely dotted, opacity=0.5] (4f) to (4m);
    	\draw[densely dotted, opacity=0.5] (3f) to (3m);
    	\draw[col5, bend right=38, looseness=1.0333333333333334] (4m) to (0m);
    	\draw[col2] (4m) to (3m);
    	\draw[col4, bend right=52, looseness=1.0666666666666667] (4m) to (5m);
    	\draw[col2] (3m) to (0m);
    	\draw[col5, bend left=38, looseness=1.0333333333333334] (3m) to (5m);
    	\draw[col1] (0m) to (5m);
    	\draw[col0, bend right=38, looseness=1.0333333333333334] (5m) to (1m);
    	\draw[col6] (5m) to (2m);
    	\draw[densely dotted] (4m) to (4e);
    	\draw[densely dotted, thick, absolute, in=-80, out=-100] (3m) to (3e);
    	\draw[densely dotted] (0m) to (0e);
    	\draw[densely dotted] (5m) to (5e);
    	\draw[densely dotted] (2m) to (2e);
    	\draw[densely dotted] (1m) to (1e);
    	\draw[col5, bend left=38, looseness=1.0333333333333334] (4e) to (0e);
    	\draw[col2] (4e) to (3e);
    	\draw[col4, bend left=52, looseness=1.0666666666666667] (4e) to (5e);
    	\draw[col2] (3e) to (0e);
    	\draw[col5, bend right=38, looseness=1.0333333333333334] (3e) to (5e);
    	\draw[col1] (0e) to (5e);
    	\draw[col0, bend left=38, looseness=1.0333333333333334] (5e) to (1e);
    	\draw[col6] (5e) to (2e);
    	\draw[densely dotted, opacity=0.5] (0e) to (0ee);
    	\draw[densely dotted, opacity=0.5] (5e) to (5ee);
    	\draw[densely dotted, opacity=0.5] (2e) to (2ee);
    	\draw[densely dotted, opacity=0.5] (1e) to (1ee);
    	\draw[densely dotted, opacity=0.5] (4e) to (4ee);
    	\draw[densely dotted, opacity=0.5] (3e) to (3ee);
    	\draw[col5, opacity=0.5, bend left=38, looseness=1.0333333333333334] (4ee) to (0ee);
    	\draw[col2, opacity=0.5] (4ee) to (3ee);
    	\draw[col4, opacity=0.5, bend right=52, looseness=1.0666666666666667] (4ee) to (5ee);
    	\draw[col2, opacity=0.5] (3ee) to (0ee);
    	\draw[col5, opacity=0.5, bend right=38, looseness=1.0333333333333334] (3ee) to (5ee);
    	\draw[col1, opacity=0.5] (0ee) to (5ee);
    	\draw[col0, opacity=0.5, bend right=38, looseness=1.0333333333333334] (5ee) to (1ee);
    	\draw[col6, opacity=0.5] (5ee) to (2ee);
     \draw (0.75, -0.4) node[fill=white,draw=white,rectangle,inner sep=1pt] (cap) {Two middle copies ($G^6_{82}$)};
    \end{tikzpicture}
}

\title{A Systematic Approach to Crossing Numbers of Cartesian Products with Paths}
 \titlerunning{A Systematic Approach to Crossing Numbers of Cartesian Products with Paths}

\author{Zayed Asiri}
{Flinders University, Adelaide, Australia}
{asir0030@flinders.edu.au}
{}
{}

\author{Ryan Burdett}
{Flinders University, Adelaide, Australia}
{ryan.burdett@flinders.edu.au}
{}
{}

\author{Markus Chimani}
{Theoretical Computer Science, Osnabrück University, Osnabrück, Germany}
{markus.chimani@uos.de}
{https://orcid.org/0000-0002-4681-5550}
{}

\author{Michael Haythorpe}
{Flinders University, Adelaide, Australia}
{michael.haythorpe@flinders.edu.au}
{https://orcid.org/0000-0001-8143-6583}
{Corresponding author}

\author{Alex Newcombe}
{Flinders University, Adelaide, Australia}
{alex.newcombe@flinders.edu.au}
{https://orcid.org/0000-0002-4401-0951}
{}

\author{Mirko H.~{Wagner}}
{Theoretical Computer Science, Osnabrück University, Osnabrück, Germany}
{mirko.wagner@uos.de}
{https://orcid.org/0000-0003-4593-8740}
{}

\authorrunning{Z.~Asiri, R.~Burdett, M.~Chimani, M.~Haythorpe, A.~Newcombe, M.~H.~Wagner}

\Copyright{Zayed Asiri, Ryan Burdett, Markus Chimani, Michael Haythorpe, Alex Newcombe, Mirko H.~Wagner} %



\category{} %

\relatedversion{} %

\acknowledgements{I want to thank \dots}%

\EventEditors{...}
\EventNoEds{0}
\EventLongTitle{...}
\EventShortTitle{...}
\EventAcronym{...}
\EventYear{...}
\EventDate{...}
\EventLocation{...}
\EventLogo{}
\SeriesVolume{0}
\ArticleNo{0}

\begin{document}
\nolinenumbers

\maketitle

{\abstract Determining the crossing numbers of Cartesian products of small graphs with arbitrarily large paths has been an ongoing topic of research since the 1970s. Doing so requires the establishment of coincident upper and lower bounds; the former is usually demonstrated by providing a suitable drawing procedure, while the latter often requires substantial theoretical arguments. Many such papers have been published, which typically focus on just one or two small graphs at a time, and use ad hoc arguments specific to those graphs. We propose a general approach which, when successful, establishes the required lower bound. This approach can be applied to the Cartesian product of any graph with arbitrarily large paths, and in each case involves solving a modified version of the crossing number problem on a finite number (typically only two or three) of small graphs. We demonstrate the potency of this approach by applying it to Cartesian products involving all 133 graphs $G$ of orders five or six, and show that it is successful in 128 cases. This includes 60 cases which a recent survey listed as either undetermined, or determined only in journals without adequate peer review.}

\section{Introduction}\label{sec-Introduction}
In this paper, we will refer to a number of common graph families, and so we list them upfront to aid the reader. $P_n$, $C_n$, $S_n$ is the path, cycle, star with $n$ edges and thus on $n+1$, $n$, $n+1$ vertices respectively. The Cartesian product of two graphs $G$ and $H$ is written as $G \Box H$. What results is a graph with vertex set $V(G) \times V(H)$, and edges between vertices $(u,u')$ and $(v,v')$ if and only if either $u = v$ and $(u',v') \in E(H)$, or $u' = v'$ and $(u,v) \in E(G)$.%

A {\em drawing} $D$ of a graph $G$ is a representation of $G$ in the plane, such that each vertex is mapped to a discrete point, and each edge $(a,b)$ is mapped to a closed curve between the points corresponding to vertices $a$ and $b$, which does not intersect with a point corresponding to any other vertex. If two curves intersect in such a drawing, we say that there is a {\em crossing} between their corresponding edges, and we denote the number of crossings in the drawing $D$ as $cr_D(G)$. Then the {\em crossing number problem} (CNP) is to determine $\crg(G) \coloneqq \min_D cr_D(G)$, the minimum number of crossings among all possible drawings of $G$.
We can assume that all intersections are crossings, rather than tangential, and also that three edges never intersect at a common point. Furthermore, a drawing is said to be {\em good} if edges incident with a common vertex do not intersect, and no two edges intersect more than once. It is well known that for any graph $G$, $\crg(G)$ crossings are achieved by a good drawing.

CNP is known to be NP-hard \cite{gareyjohnson1983}, and is notoriously difficult to solve, even for relatively small graphs. Indeed, $\crg(K_{13})$ and $\crg(K_{9,9})$ are still unknown despite significant effort. Nonetheless, there are some infinite families of graphs for which the crossing numbers are known; these are summarised in the recent survey paper by Clancy et al.~\cite{clancyetal2019}. The most prolific of these are families which result from Cartesian products, the study of which originated with a conjecture by Harary et al. \cite{hararykainen1973} in 1973 that $\crg(C_m \Box C_n) = (m-2)n$ for $n \geq m \geq 3$. Despite significant effort \cite{beinekeringeisen1980,deanrichter1995,richterthomassen1995,andersonetal1996,andersonetal1996_2,klescetal1996,richtersalazar2001,adamssonrichter2004,glebskysalazar2004} to date the conjecture has only been resolved for $m \leq 7$, and for $n \geq m(m+1)$. While resolving the $m = 4$ case, Beineke and Ringeisen \cite{beinekeringeisen1980} considered $\crg(G \Box C_n)$ for all six non-isomorphic connected graphs $G$ of order four, and successfully determined all of them except for the case $G = S_3$. This latter case was subsequently settled by Jendrol' and \v{S}cerbov\'{a} \cite{jendrolscerbova1982}. Following this, Kle\v{s}\v{c} \cite{klesc1994} determined $\crg(G \Box P_n)$ and $\crg(G \Box S_n)$ for the same six connected graphs $G$ of order four. This was then followed by a decade-long effort by Kle\v{s}\v{c} \cite{klesc1991,klesc1995,klesc1996,klesc1999,klesc1999_2,klesc2001_2} to determine $\crg(G \Box P_n)$ for all 21 non-isomorphic connected graphs $G$ of order five, which was finally completed in 2001. In the two decades since, significant efforts have been made by multiple authors (including Kle\v{s}\v{c}) to extend these results to the 112 non-isomorphic connected graphs of order six. Currently, slightly less than half of these have been resolved; the progress is chronicled in \cite{clancyetal2019}.

In order to determine $\crg(G \Box P_n)$ for a given graph $G$, one needs to establish lower and upper bounds, and show that they coincide. Establishing a valid upper bound is generally straight-forward, and is usually either achieved by providing a drawing procedure which results in the desired number of crossings, or else uses an existing result and the monotonicity property $\crg(G \Box P_n) \leq cr (H \Box P_n)$ if $G$ is a subgraph of $H$. Establishing a lower bound is typically much more complicated, usually involving ad hoc arguments specific to the graph in question. Due to this ad hoc nature, it is not uncommon for publications to focus on a single graph $G$ and determine $\crg(G \Box P_n)$ for that one case. Even in papers which focus on several graphs, often the complicated arguments are needed for only one or two graphs, and the remaining results follow as corollaries.

In this paper, we propose a new approach to determining, $\crg(G \Box P_n)$ which can be applied to any graph $G$. There are two possible outcomes to this approach; either the required lower bound is established, or else nothing is established. We will demonstrate that the approach is successful in resolving 128 of the 133 non-isomorphic connected graphs of orders five or six, including 60 cases which were hitherto unresolved. The approach we describe requires a modified version of CNP to be solved for $G \Box P_d$, where $d$ is a small number (typically $d = 2$ or $d = 3$ suffices). As such, this approach is tractable for sufficiently small graphs $G$.

The remainder of this paper is laid out as follows. In \Cref{sec-form} we introduce the modified version of CNP, which we call {\em binary-weighted capacity-constrained CNP} (BCCNP). In \Cref{sec-approach2} we describe the approach which, when successful, uses this modified version of CNP to obtain the required lower bounds for $\crg(G \Box P_n)$. In \Cref{sec-upperbounds} we berifly discuss the manner in which we will establish the required upper bounds and base cases, and also the manner in which we solve instances of BCCNP. In \Cref{sec-calculations} we report on calculations which test the efficacy of the approach, and show that we are able to determine $\crg(G \Box P_n)$ for 60 new graphs $G$ on six vertices, and also explore its efficacy on graphs $G$ with seven vertices. Finally, we conclude the paper in \Cref{sec-conclusions}.

\section{Binary-weighted capacity-constrained CNP (BCCNP)}\label{sec-form}

For a given graph $G$, denote by $C(G) \subseteq E(G) \times E(G)$ the set of all possible pairs of edges between which a crossing may occur in a good drawing. That is, $C(G)$ contains all pairs of non-adjacent edges.
Then, for a good drawing $D$ of $G$, and $c \in C(G)$, define a binary variable $x^D_c$ to be equal to 1 if crossing $c$ occurs in $D$, and 0 otherwise. CNP is then equivalent to
\[\min_D \;\;\sum_{c \in C(G)} x^D_c\mbox{, over all good drawings }D.\]

The above formulation considers every possible crossing that could occur in $D$ and then counts those that do. However, suppose that we additionally define some subset $B \subseteq C(G)$, and are only interested in good drawings that minimize the crossings in $B$:
\[\min_D \;\;\sum_{c \in B} x^D_c\mbox{, over all good drawings }D.\]
This is equivalent to the {\em weighted CNP} (cf.\ Schaefer \cite{schaefersurvey}) when all the weights are binary.

Another variant that could be considered would be to add some additional constraints to the formulation. Suppose that we have a family $\crsubset=\{\crsubset_i\}_i$ which contains subsets $\crsubset_i \subseteq C(G)$, and an $|\crsubset|$-dimensional vector $\capacities = (\capacity)$ of non-negative integer capacities. Then a set of constraints of the form 
$\sum_{c \in \crsubset_i} x^D_c \leq \capacity,$
for all $1 \leq i \leq |\crsubset|$, could be imposed: we only allow good drawings that, for each $i$, contain no more than $\capacity$ crossings from the subset $\crsubset_i$. We say  $\crsubset_i$ has a {\em maximum capacity} of $\capacity$.

The two concepts above can be combined together; we refer to the resulting problem as {\em binary-weighted capacity-constrained CNP} (BCCNP), formulated as:
\begin{alignat}{2}
	\mathrlap{\min_D \sum_{c \in B} x^D_c}\qquad \notag\\
    \mbox{s.t.\ }&D\mbox{ is a good drawing, and}\notag\\
	&\sum_{c \in \crsubset_i} x^D_c \leq \capacity  &&\forall i = 1, \hdots, |\crsubset|\label{eq:capconstr}
\end{alignat}

An instance of BCCNP requires us to specify not only the graph $G$, but also the subset~$B$, the family~$\crsubset \subseteq 2^{C(G)}$ of subsets, and the capacity vector $\capacities$. Thus, let $\crfunction (G,B,\crsubset,\capacities)$ denote the solution of a BCCNP instance. Naturally, $\crfunction (G,C(G),\emptyset,\emptyset)$ yields the standard CNP.
BCCNP has potential applications in its own right.
\delConf{For instance, one could imagine designing some kind of network where there is a maximum number of crossings that can be accommodated (capacities) in particular sections of the network, corresponding to the subsets in $\crsubset$. As long as these capacities are not exceeded, the resulting crossings are only problematic or expensive in certain parts of the network, corresponding to those in~$B$.}However, our reason for defining %
it
is that it will assist us in determining the crossing numbers for various families of graphs.

\section{Using BCCNP to determine crossing numbers} \label{sec-approach2}

Suppose that we have a graph $G$, and that we have reason to believe there are integers $s,a,b$ (with $s,a>0$) such that $\crg(G \Box P_n) = an - b$ for $n \geq s$. Certainly this is the case for every established result to date \cite{clancyetal2019}. Furthermore, suppose that we already possess a proof that this formula holds for $n = s, s+1, \hdots, t$ for some integer $t$, and also that we have established the corresponding upper bound $\crg(G \Box P_n) \leq an - b$ for $n \geq s$. Then all that remains is to determine the lower bound, $\crg(G \Box P_n) \geq an - b$ for $n > t$.

We begin by giving some definitions and notations that will be useful in the upcoming discussion. Note that the graph $G \Box P_n$ contains $n+1$ copies of $G$, with consecutive copies linked together by some edges. Denote the copies of $G$ as $G^0$, $G^1$, $\hdots$, $G^n$, each with edges $E(G^i)$. We call remaining edges \emph{path edges}; those linking $G^i$ to $G^{i+1}$ form the subgraph $H^i$, for $i = 0, \hdots, n-1$. See \cref{fig:cartesianproduct} for an example.

\adrawingFig

\begin{definition}
Consider a positive integer $a$. A good drawing of $G \Box P_n$ is \emph{\arestricted} (cf.\ \cref{fig:adrawing}) if each copy of $G$ has fewer than $a$ crossings on its edges.
\end{definition}
For sets of crossings, we use the shorthand $H' \crosstimes H'' \coloneqq E(H') \times E(H'')$ and define \[\crfunction_{a} (G,d,B) \coloneqq \crfunction (G\Box P_d, B, \Big\{G^i \crosstimes G \Box P_d\Big\}_{0 \leq i \leq d}, (a-1)\cdot {\bf 1}),\] which is the minimum number of crossings on the set of edge pairs $B$ over all \adrawing s of $G \Box P_d$.

\begin{lemma}Suppose that there is an integer $n \geq 2$ such that $\crg(G \Box P_{n-1}) \geq a(n-1) - b$, but $\crg(G \Box P_n) < an - b$. Then every crossing-minimal drawing of $G \Box P_n$ is \arestricted.\label{lem-adrawing}\end{lemma}

\begin{proof}Suppose that there is a crossing-minimal drawing $D_n$ of $G \Box P_n$ which is not \arestricted. Recall that $G \Box P_n$ has $n+1$ copies of $G$. Since $D_n$ is not \arestricted, at least one of those copies, say $G^i$, must have $a$ or more crossings on its edges. By deleting the edges $E(G^i)$ from the drawing, we obtain a drawing of a new graph with fewer than $a(n-1) - b$ crossings. However, since this new drawing is of a graph that is either homeomorphic to $G \Box P_{n-1}$, or else contains $G \Box P_{n-1}$ as a subgraph, this violates the initial assumption.
\delConf{Hence, all crossing-minimal drawings of $G \Box P_n$ must be \arestricted.}
\end{proof}

Now, consider the graph $G \Box P_d$ for some $d \geq 2$.
\delConf{In a good drawing for this graph, the set of all possible edge-crossings is $C(G \Box P_d)$.}
The arguably most natural approach when discussing Cartesian products with paths is to try to consider 
the graph ``copy-wise'', i.e., look at the graph in chunks of subgraphs $G_i\cup H_i$, possibly akin to how we define our $a$-restrictions. However, it turns out that a different way to group crossings is more useful in our proofs.
We partition the possible crossings $C(G \Box P_d)$ into $d+1$ {\em crossing bands} $A^d_i$: %
\[
        A^d_i \coloneqq \displaystyle\bigcup_{j = 0}^{2i} \Big((G^{j} \cup H^{j}) \crosstimes (G^{2i-j-1} \cup H^{2i-j-1} \cup G^{2i-j})\Big) \cup \Big(H^{j} \crosstimes H^{2i-j-2}\Big)\mbox{, for } i \in \{0, \hdots, d\},
\]
where any sets with superscripts less than zero or larger than $d$ are treated as empty sets. The partition of crossings into the crossing bands is visualized in \cref{tbl-force}. Observe that any given crossing in a good drawing of $G \Box P_d$ belongs to exactly one crossing band. For the sake of convenience, in the following discussions we will often refer to a drawing as having crossing bands; technically, this refers to the crossing bands of the underlying graph.

\begin{figure}
	\centering
	\scalebox{0.83}{
		\crossingBandsTbl
	}
	\caption{Crossing bands $A^d_i$ for $G\Box P_d$, visually highlighted as colored diagonal stripes. A colored cell represents crossings between edges of the subgraphs corresponding to the row and the column, respectively. The number ``$i$'' in a cell indicates that these crossings are contained in~$A^d_i$. Lighter shaded cells indicate redundant associations. In contrast, the restriction of an \adrawing~on copy $G^i$ sums over all crossings in row and column $G^i$; this is marked in bold black lines for~$G^3$. Note that the latter contains a very different set of crossings than does $A^d_3$ (blue cells labeled ``3'').}
	\label{tbl-force}
\end{figure}

We will show that each crossing band contributes a specific minimum number of crossings to the overall crossing number of $G\Box P_d$. Intuitively speaking, one may think about a natural drawing, like \Cref{fig:cartesianproduct}: Essentially each $G$-copy is drawn identically and requires a certain number of crossings, except for the left- and right-most $G$-copy (``front'' and ``end'') which may require less crossings. However, for more complex $G$, the special front and end may consist of multiple $G$-copies each. All path-products for which the crossing number is known allow such a ``natural'' optimal drawing, however, our proof may of course not assume that this would always be the case. Nonetheless, this pattern allows us to establish certain minimal numbers of crossings that are \emph{forced} along the crossing bands.

	Let $d \geq 2$ and $0 < m < d$.	We use the term {\em middle force} to refer to the minimum possible number of crossings from $A^d_m$ in any \adrawing~of $G \Box P_d$. Likewise, the term \emph{front force} (\emph{end force}) refers to the minimum possible number of crossings from the first $m$ (final $d-m$, respectively) crossing bands in any \adrawing~of $G \Box P_d$
	(cf.\ \cref{fig:forces}):
	\[
		\begin{array}{lllccc}
			\emph{front force}\phantom{spaceee} & \forceF & \coloneqq & \crfunction_{a}(G, d, & A^d_0 \cup \dots \cup  A^d_{m-1} & ), \\[.5em]
			\emph{middle force}\quad            & \forceM & \coloneqq & \crfunction_{a}(G, d, & A^d_m                            & ), \\[.5em]
			\emph{end force}\quad               & \forceE & \coloneqq & \crfunction_{a}(G, d, & A^d_{m+1} \cup \dots \cup  A^d_d & ).
		\end{array}
	\]

	\begin{figure}[tb]
		\forcesFig
		\caption{Crossing bands considered in a (plus-)force triple: \textsf{\textbf{(left)}} front, \textsf{\textbf{(right)}} middle and end.}
		\label{fig:forces}
	\end{figure}

	Together they define the 
 \emph{\hexaforceName}~$(\forceF, \forceM, \forceE)$.
    The rationale for considering the \hexaforceName~is revealed in the following lemma.

    \begin{lemma}Let $(\forceF, \forceM, \forceE)$ be a \hexaforceName~of $G$. Let $D_n$, with $n \geq d$,  be an \adrawing~of $G \Box P_n$. Then there are
    \begin{enumerate}\item[(a)] at least $\forceF$ crossings in $D_n$ from the first $m$ crossings bands; 
    \item[(b)] at least $\forceM$ crossings in $D_n$ from the $(m+i)$-th crossing band for $i \in \{0, \hdots, n-d\}$; and
    \item[(c)] at least $\forceE$ crossings in $D_n$ from the final $d-m$ crossing bands.\end{enumerate}\label{lem-forces}\end{lemma}

    \begin{proof}First, note that $D_n$ contains $n-d+1$ subdrawings of $G \Box P_d$. In particular, we denote the $j$-th subdrawing as $D_n^j$, for $j \in \{0, \hdots, n-d\}$; it is the one containing copies $G^{j}$ through to $G^{j+d}$. Since $D_n$ is an \adrawing, each of these subdrawings is also \arestricted. We can obtain lower bounds on how many crossings in $D_n$ come from each crossing band, by looking at the forced number of crossings in corresponding subdrawings.

    Consider $D_n^{0}$, and recall that it is \arestricted. Then, by definition, the number of crossings in $D_n^{0}$ from its first $m$ crossing bands is at least $\forceF$. However, the set of potential crossings contained in its first $m$ crossing bands is a strict subset of the potential crossings contained in the first $m$ crossing bands of $D_n$. Hence, item (a) is true. Analogous arguments hold for item (b) by considering the $m$-th crossing bands in $D_n^{i}$ for $i \in \{0, \hdots, n-d\}$, and for item (c) by considering the final $d-m$ crossing bands in $D_n^{n-d}$.\end{proof}

    \Cref{lem-adrawing,lem-forces} lead to our first main statement.

    \begin{theorem}
		\label{cor-force}
		Let $(\forceF, \forceM, \forceE)$ be a \hexaforceName~of $G$. %
  If
	\begin{align*}
		\crg(G \Box P_{d-1}) &\geq a(d-1) - b,&&&%
		\forceF + \forceE &\geq a(d-1) - b,&&\text{ and}&\forceM &\geq a,
	\end{align*}
	then for all $n \geq d$ we also have
	\[\crg(G \Box P_n) \geq an - b.\]
 \end{theorem}
    \begin{proof}Suppose that there is some value $n \geq d$ such that $\crg(G \Box P_{n-1}) \geq a(n-1) - b$, but $\crg(G \Box P_n) < an - b$. Let $D_n$ be a crossing-minimal drawing of $G \Box P_n$. By \Cref{lem-adrawing}, $D_n$ is \arestricted, and we note that it has fewer than $an - b$ crossings. Then, by \Cref{lem-forces}, $D_n$ must contain at least $\forceF + \forceE + (n-d+1)\forceM \geq an - b$ crossings, contradicting our initial assumption and completing the proof.\end{proof}

    As will be demonstrated in \Cref{sec-calculations}, \Cref{cor-force} is already sufficient to provide the desired lower bounds in several cases. However, in many cases the \hexaforceName~values are not large enough to meet the conditions of \Cref{cor-force}. In \Cref{sec-force} we briefly discuss why this might occur, and then propose a more sophisticated way of considering the forces which leads to significantly improved results.

    \subsection{Plus-forces and star-forces\label{sec-force}}

    The above results are obtained by partitioning all possible crossings of $G \Box P_n$ into various crossing bands, and then considering lower bounds on the number of crossings from those crossing bands. If each of these individual lower bounds is sufficiently large, we obtain the desired result. However, in practice, it is often possible to draw a graph such that there are relatively few crossings in one crossing band, at the cost of a higher number of crossings in nearby crossing bands.
    \Cref{cor-force} does not take this into account and pessimistically assumes that the minima for each crossing band could be attained simultaneously. Whenever this is not the case, the conditions of \Cref{cor-force} are not met and it cannot be applied.

    To remedy this, we propose a more sophisticated way of measuring forces. When calculating the above force values by solving a BCCNP instance, the objective function only considers crossings from certain crossing band(s); except for insisting on an \adrawing, there are no further restrictions on the other crossing bands. 
    This leads to the issue mentioned above, where additional crossings are able to ``hide'' in the other, largely unrestricted, crossing bands. With this in mind, we introduce {\em plus-forces}, which are illustrated in \cref{fig:forces}. Each plus-force is defined the same as its corresponding force, but with additional restrictions that are placed on the number of crossings in other crossing bands. In particular, for the middle and end plus-forces, restrictions are placed on exactly one, adjacent, crossing band. For the front plus-force, restrictions are placed on $\numFPF$ crossing bands, for $\numFPF \in \{0, \hdots, d-m+1\}$. Hence, we use the terms $\forcePlusF, \forcePlusM, \forcePlusE$ to denote the front, middle, and end plus-forces, respectively. We again require $d\geq 2$ and $0<m<d$:

    The \emph{middle plus-force} $\forcePlusM$ is the minimum number of crossings from $A^d_m$ in any \adrawing, such that there are at most $a$ crossings from $A^d_{m+1}$. This can be computed analogously to $\forceM$, by simply adding the one extra capacity constraint
$\sum_{c \in A^d_{m+1}} x^D_c \leq a.$
    
	The \emph{end plus-force} $\forcePlusE$ is the minimum number of crossings from the final $d-m$ crossing bands in any \adrawing, such 
 that there are at most $a-1$ crossings from $A^d_m$. This can be computed in the same way as $\forceE$ by adding
$\sum_{c \in A^d_{m}} x^D_c \leq a-1$. We specifically clarify that the right-hand side here is really $a-1$, not $a$ as in the middle plus-force, or a multiple of $a$ as in the front plus-force below.

For the \emph{front plus-force} $\forcePlusF$, we need to specify $\numFPF \in \{0, \hdots, d-m+1\}$, which is the number of crossing bands (starting from $A^d_m$) that restrictions should be placed on. Then, $\forcePlusF$ is the minimum number of crossings from the first $m$ crossing bands in any \adrawing, such
that, for every $r \in \{0,1,\hdots,\numFPF-1\}$, there are at most $(r+1)\cdot a$ crossings in $\bigcup_{i = m}^{m+r} A^d_i$. Hence, $\forcePlusF$ can be computed in the same way we compute $\forceF$ by adding $\numFPF$ additional capacity constraints:
	\begin{align*}
		\sum_{i = m}^{m+r} \sum_{c \in A^d_{i}} x^D_c \leq (r+1) \cdot a,&&&\forall r \in \{0,1, \hdots, \numFPF-1\}.
	\end{align*}
Note that if $\numFPF = 0$, no additional restrictions are placed, and hence $\forcePlus^{<m}_{d,a}(0) = \forceF$.

    Since the plus-forces are simply more restricted versions of the forces, each plus-force is at least as large as its corresponding force.
    Whether or not they coincide tells us something about \adrawing s of the underlying graph.
    For instance, if $\forceM < \forcePlusM$, it tells us that in order for an \adrawing~of $G \Box P_d$ to have only $\forceM$ crossings from $A^d_m$, it must have more than $a$ crossings from $A^d_{m+1}$.
    To capture this kind of information, we define the following {\em star-forces}:
 \[
		\begin{array}{llllrlrr}
			\emph{front star-force}\phantom{spaceee} & \forceStarF & \coloneqq \min \{ & \forcePlusF & ,    & ~\forceF & +\ 1 & \} \\[.5em]
			\emph{middle star-force}\quad            & \forceStarM & \coloneqq \min \{ & \forcePlusM & ,    & ~\forceM & +\ 1 & \} \\[.5em]
			\emph{end star-force}\quad               & \forceStarE & \coloneqq \min \{ & \forcePlusE & -\ 1\ , & ~\forceE &     & \}
		\end{array}
   \]

\begin{lemma}Let $(\forceStarF, \forceStarM, \forceStarE)$ be a \hexaforceStarName~of $G$, and suppose that $\forceStarM \geq a$. Let $D_n$ be an \adrawing~of $G \Box P_n$ for $n \geq d$. Let $z=1$ if $D_n$ contains exactly $a-1$ crossings from $A^n_{n-d+m}$, and $z=0$ otherwise. Then:
\begin{enumerate}\item[(a)] $D_n$ contains at least $\forceStarF - 1$ crossings from the first $m$ crossing bands;
\item[(b)] $D_n$ contains at least $a-1$ crossings from $A^n_{i+m}$ for $i \in \{0, \hdots, n-d\}$;
\item[(c)] if $D_n$ contains exactly $a-1$ crossings from $A^n_{i+m}$ for $i \in \{0, \hdots, n-d-1\}$, then $D_n$ contains at least $a+1$ crossings from $A^n_{i+m+1}$; and
\item[(d)] $D_n$ contains at least $\forceStarE + z$ crossings from the final $d - m$ crossing bands.
\end{enumerate}\label{thm-smallresults}\end{lemma}

\begin{proof}As previously noted, $D_n$ contains $n - d + 1$ subdrawings of $G \Box P_d$, which we denote as $D_n^{j}$ for $j \in \{0, \hdots, n-d\}$. From \Cref{lem-forces}(a), we have that $D_n$ contains at least $\forceF$ crossings from the first $m$ crossing bands. From the definition of $\forceStarF$, it is clear that $\forceF \geq \forceStarF - 1$. Hence, item (a) is true.

By \Cref{lem-forces}(b), we have that $D_n$ contains at least $\forceM$ crossings from $A^n_{i+m}$ for $i \in \{0, \hdots, n-d\}$. Now consider $\forceStarM$, which by definition is either equal to $\forcePlusM$ or $\forceM + 1$. Since $\forcePlusM \geq \forceM$ and both are integers, there are only two possibilities; either $\forceM = \forceStarM$, or $\forceM = \forceStarM - 1$. Either way, item (b) is true. Now, suppose that $D_n$ contains exactly $a-1$ crossings from $A^n_{i+m}$, which is only possible if $\forceM = \forceStarM - 1 \geq a - 1$. This, in turn, is only possible if $\forcePlusM > \forceM$, and so $\forcePlusM \geq a$. According to the definition of $\forcePlusM$, it is hence impossible to draw $G \Box P_d$ with fewer than $a$ crossings from $A^d_m$ and fewer than $a+1$ crossings from $A^d_{m+1}$. However, in the subdrawing $D_n^{i}$ there are fewer than $a$ crossings from $A^d_m$. Hence, there must be at least $a+1$ crossings in $D_n^{i}$ from $A^d_{m+1}$, which implies that there are at least $a+1$ crossings in $D_n$ from $A^n_{i+m+1}$. Therefore item (c) is true.

Finally, by \Cref{lem-forces}(c), we have that $D_n$ contains at least $\forceE$ crossings from the final $d-m$ crossing bands. Note that by definition we have $\forceE \geq \forceStarE$, so if $z = 0$ then item (d) is true. Suppose that $z = 1$, and item (d) is false. This could only be the case if $\forceE = \forceStarE$, which implies that $D_n$ must have exactly $\forceStarE$ crossings from the final $d-m$ crossing bands, i.e., one crossing less than suggested by item (d). However, since $z = 1$, $D_n$ contains exactly $a-1$ crossings from $A^n_{n-d+m}$ by definition. Thus, the subdrawing $D_n^{n-d}$ must also have no more than $a-1$ crossings in its $m$-th crossing band, and since $D_n^{n-d}$ is an \adrawing, it must have at least $\forcePlusE$ crossings from its final $d-m$ crossing bands. This implies that $D_n$ must contain at least $\forcePlusE$ crossings from its final $d - m$ crossing bands. However, by definition we have $\forcePlusE \geq \forceStarE + 1 = \forceStarE + z$. Hence item (d) is true.\end{proof}

\Cref{thm-smallresults} enables us to establish a stronger version of \Cref{cor-force}. 

	\begin{theorem}
		\label{thm-force-mf-star}
		Let $(\forceStarF, \forceStarM, \forceStarE)$ be a \hexaforceStarName~of $G$ and $\numFPF' \coloneqq \max(\numFPF, 1)$.
		If
		\begin{align*}
			\crg(G \Box P_{d+\numFPF'-2}) &\geq a(d+\numFPF'-2) - b,&&&%
			\forceStarF + \forceStarE &\geq a(d-1) - b,&&\text{ and}&\forceStarM &\geq a
		\end{align*}
		then for $n \geq d+\numFPF'-1$,
		\[\crg(G \Box P_n) \geq an - b.\]
	\end{theorem}

\begin{proof}
	Suppose that there is some value $n \geq d+\numFPF'-1$ such that $\crg(G \Box P_{n-1}) \geq a(n-1) - b$, but $\crg(G \Box P_n) < an - b$. Let $n$ be minimal with this property. Let $D_n$ be a crossing-minimal drawing of $G \Box P_n$, then $D_n$ contains fewer than $an - b$ crossings, and from \Cref{lem-adrawing} it is an \adrawing. We want to count the number of crossings in $D_n$, cf.\ \Cref{fig:mainproof-numcrossings}. We will use the term {\em front crossings} to refer to the crossings in $D_n$ from the first $m$ crossing bands, {\em end crossings} to refer to the crossings in $D_n$ from the final $d-m$ crossing bands, and {\em middle crossings} to refer to all remaining crossings of $D_n$. To assist in calculating these, let $\numCrossings_i^j$ be the number of crossings in $D_n$ from $A^n_{m+i} \cup \hdots \cup A^n_{m+j}$. To simplify the notation, let $\numCrossings_i := \numCrossings_i^{n-d}$. Note that the number of middle crossings in $D_n$ is equal to $\numCrossings_0$. We seek to compute a lower bound for $\numCrossings_{i}$ for $i \in \{0, \hdots, n-d\}$. As in \Cref{thm-smallresults} it will be convenient to set $z=1$ if $D_n$ has exactly $a-1$ crossings from $A^n_{n-d+m}$, and $z=0$ otherwise.

\begin{figure}
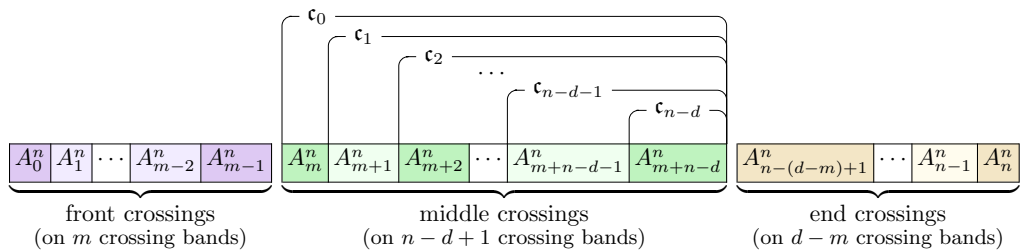

	\centering\noindent
	\scalebox{0.89}{
		\numcrossingsFig
	}
	\caption{Counting crossings in $D_n$.}
	\label{fig:mainproof-numcrossings}
\end{figure}

Recall from \Cref{thm-smallresults}(b) that $\numCrossings_{i}^{i} \geq a-1$, for $i \in \{0, \hdots, n-d\}$. If equality holds for any $i < n-d$, then from \Cref{thm-smallresults}(c) we know that $\numCrossings_{i}^{i+1} \geq (a-1) + (a+1) = 2a$. Inductively then, it is clear that $\numCrossings_{i}^{j} \geq (j-i+1)a - 1$ for $i,j \in \{0, \hdots, n-d\}$ and $i \leq j$, and if equality is met it implies that $D_n$ contains exactly $a-1$ crossings from $A^n_{j+m}$. Hence, by setting $j = n - d$, we obtain
\begin{equation}\numCrossings_{i} \geq (n+1-d-i)a - z, \qquad \forall i \in \{0, \hdots, n-d\}.\label{eq-C}\end{equation}

	In particular, setting $i = 0$ in \eqref{eq-C} tells us that there are at least $\numCrossings_0 = (n+1-d)a - z$ middle crossings, while \Cref{thm-smallresults}(a) and \Cref{thm-smallresults}(d) give us lower bounds on the front and end crossings respectively. Combining all of these, we can see that $D_n$ contains at least $(\forceStarF-1) + \numCrossings_0 + (\forceStarE + z) \geq an - b - 1$ crossings. Since by assumption $D_n$ has fewer than $an - b$ crossings, it must have exactly $an - b - 1$ crossings. In particular, $D_n$ has exactly $\forceStarF - 1$ front crossings. However, from the definition of $\forceStarF$, this is only possible if $\forcePlusF > \forceF$ and $\numFPF \geq 1$. Hence, $D_n$ must have fewer than $\forcePlusF$ front crossings, and so from the definition of \forcePlusF, there must be some positive integer $j \leq \numFPF$ such that $\numCrossings_0^{j-1} \geq aj+1$. Note that since $n \geq d+\numFPF'-1$, and $\numFPF \leq \numFPF'$, we have $j \leq n-d+1$. If $j = n-d+1$ then $\numCrossings_0 = \numCrossings_0^{j-1} \geq a(n-d+1) + 1$. Alternatively, if $j < n-d+1$, then from \eqref{eq-C} we have $\numCrossings_{j} \geq (n+1-d-j)a - z$, and so $\numCrossings_{0} = \numCrossings_{0}^{j-1} + \numCrossings_{j} \geq (n-d+1)a - z + 1$.Either way, adding $\numCrossings_0$ to $(\forceStarF - 1) + (\forceStarE + z)$ implies that $D_n$ contains at least $an - b$ crossings, contradicting our initial assumption and completing the proof.\end{proof}

As will be shown in \Cref{sec-calculations}, \Cref{thm-force-mf-star} is much stronger than \Cref{cor-force}, and will enable us to established the desired lower bounds for many more graph classes.

\section{Upper bounds, base cases, and computing forces by solving BCCNP}
\label{sec-upperbounds}
\subparagraph{Upper bounds.}
An important ingredient of the crossing number proofs is establishing an upper bound matching the lower bound. In practice, in the case of $G\Box P_n$, this turns out to be surprisingly simple, as simple enumeration schemes for simple drawings suffice. In fact, in most cases all but the very first and the very last copy of $G$ can be drawn identically, and these drawings are easy to figure out. Thus, we will not discuss this further in the following, but state that whenever we could prove a lower bound, we also identified a construction for the matching upper bound. The tables in the appendix show the corresponding subdrawings (and explain their interpretation).

\subparagraph{Base cases and force computation.}
To apply our above theorems, each graph class $G\Box P_n$ requires us to establish (a) $\crg(G\Box P_n)$ for small values of $n$ as the base cases, and (b) $G$-specific force values: sound proofs of BCCNP for small values $d,m$. Each thereby considered graph is relatively small (typically below $50$ edges), but since our goal is to solve many different graph classes, there are too many such instances to reasonably do them all by hand.

It was shown in Buchheim et al. \cite{buchheimetal2008} that CNP can be formulated as an integer linear program (ILP). Subsequently, Chimani et al. \cite{chimanietal2009,chimanietal2008} implemented this formulation as a practical algorithm, which requires sophisticated column generation techniques and other speed-up heuristics, to be tractable for ``real-world'' graphs with up to 100 edges.
As this implementation is generally not easy to understand and validate, Chimani and Wiedera \cite{chimaniwiedera2016} presented a proof and validation framework that extracts a crossing number proof from the ILP computation; this certificate can be independently validated (with comparably simple methods). This tool has since regularly been used to establish base cases or to validate results, e.g.,~\cite{clancyetal2019,bokaletal2022,clancyetal2020,bokalleanos2018,stastimkova2023,klescetal2017,stas2020,stas2020_2,huhnikkuni2024}.
The modifications necessary to tackle BCCNP instead of CNP are relatively straight-forward: analogously to the case of the \emph{simultaneous crossing number}~\cite{chimanietal2008_2}, when restricting the objective function to disregard certain crossings, these crossing variables are considered with a small enough coefficient $\varepsilon$ instead of $0$, to avoid issues with non-good drawings. The capacity constraints \eqref{eq:capconstr} can directly be added to the ILP computation and the proof validator as well. We call this algorithm the \emph{BCCNP solver}; all corresponding certificates for the crossing number proofs required below (as well as the proof validator), can be found at~\cite{results-webpage}.

\section{Calculations}\label{sec-calculations}

\Cref{cor-force,thm-force-mf-star} are only useful when the conditions are met, and checking if they are met requires us to compute multiple forces using the BCCNP solver for some setting of the parameters. Clearly, there are three possible outcomes to this approach. The first outcome is that the values computed by the BCCNP solver satisfy the requirements of at least one of the two theorems, and the desired result is obtained. The second outcome is that the values computed by the BCCNP solver do not satisfy the requirements of either theorem, and no result is obtained. The third outcome is that the BCCNP solver is unable to compute the values, as the graph is too large or complex to be tackled with reasonable computing power\footnote{All computations were conducted on a Intel Xeon Gold 6134 with a time limit of 2 hours.}.

To test the efficacy of the two proposed approaches, we ran them both on all 21 non-isomorphic connected graphs on five vertices, and all 112 non-isomorphic connected graphs on six vertices.
Since \Cref{cor-force,thm-force-mf-star} attempt to show that $\crg(G \Box P_n) \geq an - b$, we began by assuming that $\crg(G \Box P_n) = an - b$ for some integers $a, b$.
We computed $\crg(G \Box P_n)$ for some small values of $n$, and used these to predict the values of $a,b$. Then, we checked to see if they matched the upper bound yielded from our simple enumeration schemes. Indeed, in all tested cases, the values of $a,b$ matched.

Then, armed with an appropriate choice of $a,b$, we computed the forces for the $(d,m)$ pair $(2,1)$, and checked to see if they satisfied the relevant conditions for \cref{cor-force}. If not, we then proceeded to compute the plus-forces for the same $(d,m)$ pair and $\numFPF = d - m + 1$. Then, we checked if the relevant conditions of \cref{thm-force-mf-star} were satisfied first for $\numFPF = 0$ (by using 
$\forcePlus^{<m}_{d,a}(0) = \forceF$), and if not, then for $\numFPF = d - m + 1$. If any of these checks were successful, then we computed the relevant base case to settle the instance; note that checking for $\numFPF = 0$ first is worthwhile because the corresponding base case is considerably smaller than for $\numFPF = n - d + 1$. If none of the checks were successful, then we then repeated the above for the $(d,m)$ pairs $(3,1)$, and $(4,2)$, only declaring failure if each of these proved unsuccessful.

We summarize the results of these calculations in \cref{tab-results}.
In particular, we note that we obtained a successful result for all 21 graphs on five vertices, and 107 (out of 112) graphs on six vertices.
The successful instances include $60$ graphs on six vertices which, according to the recent survey by Clancy et al.\ \cite{clancyetal2019}, have not previously been handled in the literature\footnote{This includes instances which have not been handled at all in the literature, and also instances which have only been handled in journals or periodicals which Clancy et al. \cite{clancyetal2019} determined ``impose no peer review, or that which does occur is inadequate.'' They note that such results ``should be revisited and submitted to thorough peer review''. As such, we treat such results as new in this manuscript.}. In particular, we establish the following results:

\begin{table}
	\centering
	\begin{tabular}{>{\bfseries}r*{3}{lr@{ }>{\tiny}rl}l>{\boldmath}r>{\boldmath\tiny}r}
        \toprule
        \multicolumn{1}{l}{\boldmath$G\Box P_n$,}
        & \multicolumn{3}{c}{\bf \Cref{cor-force}}  &    
        & \multicolumn{3}{c}{\bf \cref{thm-force-mf-star}} &
        & \multicolumn{3}{c}{\bf \cref{thm-force-mf-star}} &
        & \multicolumn{3}{c}{\bf total} \\
        \multicolumn{1}{l}{where $G$ are}
        & \multicolumn{3}{c}{} && \multicolumn{3}{c}{with $\numFPF = 0$} && \multicolumn{4}{c}{with $\numFPF = d-m+1$}            & \multicolumn{3}{c}{} \\
        \midrule
        5-vertex graphs && $19$                                   & $90\%$ &&& $21$                                     & $100\%$ &&& $21$                                      & $100\%$ &&& $21$                   & $100\%$ \\
        6-vertex graphs && $91$                                   & $81\%$ &&& $104$                                    & $ 93\%$ &&& $105$                                     & $94\%$  &&& $107  $                & $96\%$  \\
        7-vertex graphs && $529$                                  & $62$\% &&& $573$                                    & $ 67$\% &&& $675$                                     & $ 79$\% &&& $676$                  & $ 79$\% \\
        \bottomrule
	\end{tabular}
	\caption{
		The number of instances (as well as percentage) for which \cref{cor-force,thm-force-mf-star} yield positive results for at least one $(d,m)$ pair $(2,1)$, $(3,1)$, or $(4,2)$.}
	\label{tab-results}
\end{table}

\begin{theorem}
Let $i$, in $G^6_i$, be the graph indices from \cite{clancyetal2019} to name all 6-vertex graphs\footnote{The index scheme from \cite{clancyetal2019} includes the 44 disconnected 6-vertex graphs, and hence goes up to 156.}.
	For large enough $n$ (in most cases $n \geq 1$), the crossing number $\crg(G^6_i \Box P_n)$ is given by: 

	{\begin{center}
    \begin{tabular}{>{$}W{l}{1.5cm}<{$}>{$}W{l}{1cm}<{$}>{$}W{r}{7cm}<{$}}
		\crg(G^6_i \Box P_n) & i \text{\ (new results)} & {\color{darkgray}\hspace{2cm} i \text{\ (previously established, reproved)}} \\\midrule
        0n+0 & \confI{25,40}\\
        1n-1 & \confI{26, 28, 41, 43, 46, 60}\\
        2n-4 & \confI{42}\\
        2n-2 & \newI{49,} \confI{27,29,44, 45, 47, 53, 54, 59, 63, 64, 66, 74, 77, 83, 94}\\
        2n+0 & \confI{61, 75, 85}\\
        3n-5 & \confI{62}\\
        3n-3 & \newI{67, 76, 78, 92,} \confI{51, 65, 70, 89, 90, 120}\\
        3n-1 & \newI{84, 95, 98, 110} \confI{68, 71, 86, 87, 91, 111} \\
        4n-4 & \confI{31, 72, 79, 80, 104, 113}\\
        4n-2 & \newI{82, 88, 99, 100, 101, 108, 115, 116, 118, 133}\\
        4n   & \newI{93, 109, 112, 121, 130, 137}\\
        5n-3 & \newI{81, 102, 106, 107, 114, 122, 123, 124, 127,} \confI{125}\\
        5n-1 & \newI{131, 138, 146}\\
        6n-4 & \newI{134}\\
        6n-2 & \newI{105, 126, 128, 129, 132, 140, 141, 143,} \confI{103}\\
        6n   & \newI{152}\\
        7n-1 & \newI{119, 135, 139, 145, 148}\\
        8n-2 & \newI{144, 147, 151}\\
        9n-3 & \newI{150}\\
        9n-1 & \newI{154}\\
        10n  & \newI{149, 153}\\
        12n  & \newI{155}\\
    \end{tabular}
\end{center}}
    The table includes all 112 connected 6-vertex graphs except for $i=31,48,73,142,156$.
\end{theorem}

The appendix lists all these graphs, together with their upper bound construction and the required values for $d,m,\numFPF$. As mentioned above, the computer generated proof files for the BCCNP base cases (together with a proof validator) can be found at~\cite{results-webpage}.
It is worth noting that in 78 of the 112 cases the smallest setting, $d=2$ and $m=1$, already suffices to prove the claim with \cref{thm-force-mf-star} and in 66 cases, even with \cref{cor-force}.

\delConf{\begin{proof}In each case we are able to establish the lower bounds using at least one of \Cref{cor-force,thm-force-mf-star}. \cref{tbl:res:six} list the parameter settings which were used to establish the lower bounds for each successfully solved graph. Proof files for each of the obtained values can be found at \cite{proofSite} along with a tool that can be used to verify the results.

Then, all that remains is to handle the upper bounds. Also contained in \cref{tbl:res:six} are small figures corresponding to each graph, indicating a drawing procedure which establishes the desired upper bound. The small figures should be interpreted as follows. To draw $G \Box P_n$ with the indicated number of crossings, $n+1$ copies of $G$ should be drawn side by side, with the path edges forming horizontal lines. We call the first and last copies of $G$ the {\em end copies}, and the remaining $n-1$ copies are the {\em middle copies}. In the small figures from \cref{tbl:res:six}, we show two drawings of $G$. The end copies should be drawn identically to the right-most drawing in the small figure (taking the mirror image for the first copy). All of the middle copies should be drawn identically to the left-most drawing in the small figure. The edge colors are chosen arbitrarily, but are consistent between the copies. Any edges which end up stubs at the top and bottom are routed around the graph and connected according to the symbol of the stub and the edge color. The figures are presented in such a way that routing the edges around the graph never needs to introduce any additional crossings. If, in the small figure, the left-most drawing contains $c_l$ crossings, and the right-most drawing contains $c_r$ crossings, then an upper bound of $c_l(n-1) + 2c_r$ is established for $\crg(G \Box P_n)$.
For the graphs with indices $i=42,62,63,82,95$ marked by a {\color{red} !} between its two copies in \cref{tbl:res:six}, additional drawings are provided in \cref{tbl:res:special:six}, because for a sufficient upper bound for these graphs we need to draw two neighbouring copies together either at the ends or in the middle.
\end{proof}}

We note that, of the five graphs on six vertices we were unable to obtain a solution, four have already been handled in the literature. In particular, the graph with index $i = 31$ was handled in Bokal \cite{bokal2007}, the graphs with indices $i = 48, 73$ were handled in Kle\v{s}\v{c} and Petrillov\'{a} \cite{klescpetrillova2013_2}, and the graph with index $i = 156$ was handled in Zheng et al. \cite{zhengetal2007}. As such, there is only one graph remaining from the six-vertex set that is yet to be handled in the literature at all, specifically the graph with index $i = 142$. It is equivalent to $K_5$ with a pendant edge attached to one of the vertices, and based on our experiments it appears the crossing number should be $9n-3$.

Finally, we also tackled the Cartesian products of paths with any connected 7-vertex graph. There are few existing crossing number results on these graph classes (nine, according to~\cite{clancyetal2019}).
Our goal here is to determine how many of these graph classes are in fact relatively simple---in terms of solvability using our framework---and thus do not merit specific treatment in the literature. On the other hand, we hope that (failings of) our approach may point at graphs which require us to identify new lower bound arguments. Overall, we see that already with only small values of $d,m,\numFPF$, our approach generates proofs for nearly 4/5 of all $G^7\Box P_n$ graph classes. Again, see the appendix for more details.

\section{Conclusions}\label{sec-conclusions}

The potency of the approach described in this manuscript is demonstrated by the number of new results that have been established, particularly given that this field of research has been saturated for decades with papers painstakingly establishing one or two results at a time.

In addition to the volume of new results, another benefit of this approach is that it lets us identify the truly difficult graphs, compared to graphs which just require more exhaustive versions of existing arguments. These difficult graphs can then be given individual attention. As shown in the previous section, our approach was successful for all but five graphs from the six-vertex set. Four of these five graphs were previously handled in the literature. Two of them ($i = 31, 156$) required specialised arguments applicable only to those graphs, while the other two ($i = 48, 73$) were corollaries. Researchers in this field can now focus on the one remaining case ($i = 142$), potentially handling not only this graph, but perhaps a family of graphs which emerge from complete graphs with pendant edges added.

Our star-forces, as strengthenings of our basic forces, have essentially accounted for cases where one crossing was able to ``hide'' in a nearby crossing band. Although this proved to be powerful, it stands to reason that there could be instances where multiple crossings are able to hide, and an accordingly stronger definition and theorem statement would be required. Indeed, an analysis of the $i = 31$ instance shows this to be exactly the case; in this graph, two crossings can hide. Given this context, our approach could perhaps be viewed as the first step in a more general approach that seeks to account for cases where arbitrarily many crossings can hide.
Furthermore, it seems likely that the BCCNP framework can be used to attack other families of graphs, particularly those resulting from graph products. In particular, we anticipate that it should be possible to obtain new results for Cartesian products with cycles, although this will likely require further fine-tuning of the crossing bands and corresponding forces. Other common families such as Cartesian products with stars, as well as other kinds of graph products such as strong products, should also be considered in this framework.

\newpage

\appendix
\renewcommand\thefigure{A\arabic{figure}}    
\setcounter{figure}{0} 
\renewcommand{\thetable}{A\arabic{table}}
\setcounter{table}{0}
\section{APPENDIX}

\cref{tbl:res:six,tbl:res:five,tbl:res:seven} list the smallest $(d,m)$ pair for which at least one of \Cref{cor-force,thm-force-mf-star} establishes the lower bounds for each successfully solved graph.
The background color intensity reflects how high the values are.
The \texttt{d,m}${}^*$ indicates that \cref{cor-force} does not suffice and that \cref{thm-force-mf-star} with $\numFPF=0$ is needed.
\texttt{d,m}${}_*^*$ indicates that \cref{thm-force-mf-star} with $\numFPF = d-m+1$ is needed.
The blueness of the background also reflects this.
Proof files for each of the obtained values can be found at \cite{rawdata} along with a tool that can be used to verify the results.

Then, all that remains is to handle the upper bounds. Also contained in \cref{tbl:res:six,tbl:res:five,tbl:res:seven} are small figures corresponding to each graph, indicating a drawing procedure which establishes the desired upper bound. The small figures should be interpreted as follows. To draw $G \Box P_n$ with the indicated number of crossings, $n+1$ copies of $G$ should be drawn side by side, with the path edges forming horizontal lines. We call the first and last copies of $G$ the {\em end copies}, and the remaining $n-1$ copies are the {\em middle copies}. In the small figures from \cref{tbl:res:six}, we show two drawings of $G$. The end copies should be drawn identically to the right-most drawing in the small figure (taking the mirror image for the first copy). All of the middle copies should be drawn identically to the left-most drawing in the small figure. The edge colors are chosen arbitrarily, but are consistent between the copies. Any edges which end up stubs at the top and bottom are routed around the graph and connected according to the symbol of the stub and the edge color. The figures are presented in such a way that routing the edges around the graph never needs to introduce any additional crossings. If, in the small figure, the left-most drawing contains $c_l$ crossings, and the right-most drawing contains $c_r$ crossings, then an upper bound of $c_l(n-1) + 2c_r$ is established for $\crg(G \Box P_n)$.
For the graphs marked by a {\color{red} !} between its two copies in \cref{tbl:res:six,tbl:res:seven}, we need to draw two neighboring copies together either at the ends or in the middle for a sufficient upper bound.
Essentially, we use one (or more) of the following three strategies:\\[-4em]

{\centering 
\hspace{\fill}
\specialFortyTwo
\hspace{\fill}
\specialSixtyThree
\hspace{\fill}
\specialEightyTwo
\hspace{\fill}
\\[1.5em]}

We use the following symbols to indicate the status of the results:\\[-4em]

{\centering
\resultsTblLegend
\pagebreak
\captionof{figure}{Results for 5-vertex graphs. We omit the planar instances $i = 1, 8$.}
\resultsFiveTbl
\label{tbl:res:five}
\vspace{2em}
\captionof{figure}{Results for 6-vertex graphs}
\resultsSixTbl
\label{tbl:res:six}
\vspace{2em}
\captionof{figure}{Successful results for 7-vertex graphs.}
\resultsSevenTbl
\label{tbl:res:seven}
\\}
\end{document}